\documentclass[12pt]{extarticle}
\usepackage{epsf}
\usepackage{amsbsy,amsmath}
\usepackage{mathtools}
\usepackage{mathrsfs}
\usepackage{amsfonts}
\usepackage{amsmath}
\usepackage{amssymb}
\usepackage{enumitem}

\usepackage{mathabx}

\usepackage{eucal}
\usepackage{enumitem}
\usepackage{graphics,mathrsfs}
\usepackage{amsthm}
\usepackage{secdot}
\usepackage{xcolor}
\newtheorem{theorem}{Theorem}[section]
\newtheorem{lemma}[theorem]{Lemma}

\theoremstyle{definition}
\newtheorem{definition}[theorem]{Definition}

\theoremstyle{remark}
\theoremstyle{definition}
\newtheorem{remark}[theorem]{Remark}
\numberwithin{equation}{section}

\numberwithin{equation}{section}

\newsavebox{\savepar}

\begin{document}	
	
	
	
	
	\title{A topological approach to an elliptic problem}
	\author{\bf Debajyoti Choudhuri$^{a}$ \footnote{Corresponding author: dc.iit12@gmail.com (Debajyoti Choudhuri)}, Vikas Jaiswal$^{b}$ \\
		\small{$^{a}$ Department of Mathematics, Indian Institute of Technology Bhubaneswar,}\\
		\small{Odisha, 752050, India} \\
	\small{$^{b}$Department of Mathematics, National Institute of Technology Rourkela,}\\
	\small{Odisha, 769008, India}	\\}
	\maketitle
	\begin{abstract}
		\noindent 	In this paper, we study an elliptic problem involving a $p$-Laplacian operator and a potential well  which is driven by a critical and singular nonlinearity. Under the limiting case of a parameter blowing up to $\infty$ yields solutions to a different problem where the effect of the potential well becomes negligible. 
			\begin{flushleft}
		\end{flushleft}
	\end{abstract}

	\section{Introduction}
	Nonlinear elliptic problems involving the interaction of singular and critical nonlinearities have attracted considerable attention due to their intricate analytical structure and their relevance in models arising from nonlinear diffusion, reaction-diffusion systems, and mathematical physics. In this paper, we study the boundary value problem
\[
\begin{cases}\label{main_prob}
-\Delta_{p} u + \mu V(x)|u|^{p-2}u = \beta |u|^{p^*-2}u + \lambda |u|^{-\gamma-1}u, & \text{in } \Omega, \\
u = 0, & \text{on } \partial \Omega,
\end{cases}
\tag{1.1}
\]
where $\Omega \subset \mathbb{R}^N (N \geq 3)$ is a bounded domain with smooth boundary, -$\Delta_{p}$ denotes a quasilinear elliptic operator, $V:\Omega \to \mathbb{R}$ is a potential, $\mu, \beta,\lambda>0, p^* = \frac{Np}{N-p},$ and $\gamma \in (0,1)$.

The presence of the critical exponent $p^*$ leads to the well-known lack of compactness in Sobolev embeddings, a difficulty first systematically addressed via the concentration-compactness principle introduced by Lions \cite{PLL1}. On the other hand, singular nonlinearities of the type $|u|^{-\gamma-1}u$ introduce additional challenges related to the degeneracy of the energy functional and the behavior of solutions near zero.

The combination of these two features produces a delicate competition between concave (singular) and convex (critical) effects. This interplay was first explored in the semilinear setting in the classical work of Ambrosetti-Brezis-Cerami \cite{ACR}, and has since been extended to quasilinear operators and singular frameworks; see, for instance, Giacomoni et al. \cite{GST1} and more recent contributions such as Faraci and Iannizzotto \cite{faraci1} Ghosh et al. \cite{GS2}, and Goyal and Sreenadh \cite{GS1}.

From a variational perspective, problem \eqref{main_prob} is challenging because the associated energy functional
is not well-defined on the whole Sobolev space due to the singular term. Moreover, the critical growth term prevents compactness of Palais–Smale sequences. To overcome these issues, one typically combines truncation techniques, Nehari manifold methods, and refined compactness arguments. In particular, the decomposition of the Nehari manifold and fibering map analysis provide an effective framework for handling concave–convex nonlinearities, see Brown and Zhang \cite{BL1}. However, in this paper our approach is based on topological methods.

	Besides existence and multiplicity issues, the qualitative properties of weak solutions are also of considerable interest. The presence of the singular term $\lambda u^{-\gamma}$ introduces additional analytical difficulties, since the nonlinearity may become unbounded near the zero set of the solution. To overcome this difficulty, we establish several regularity properties of positive weak solutions. More precisely, we first derive a weak Harnack-type estimate which yields a positive lower bound on compact subsets of $\Omega$. This estimate enables us to prove that weak solutions are essentially bounded, namely $u\in L^\infty(\Omega)$, and that the nonlinear term $	f(x)=\beta u^{p^*-1}+\lambda u^{-\gamma}-\mu V(x)u^{p-1}$
	belongs to $L^\infty_{\mathrm{loc}}(\Omega)$. Consequently, the equation can be viewed locally as a $p$-Laplace equation with a bounded right-hand side. By invoking the classical regularity results of Tolksdorf \cite{Tolksdorf1984} and Lieberman \cite{Lieberman1988, Lieberman1991}, we conclude that every positive weak solution belongs to $C^{1,\alpha}_{\mathrm{loc}}(\Omega)$ for some $\alpha\in(0,1)$.

\subsection{Main Assumptions}
We assume:
\begin{enumerate}
\item (V1):~ $V \in L^\infty(\Omega)$ such that $V(x) \geq 0 \quad \text{a.e. in } \Omega$.
\item (V2):~$\tilde{\Omega}=\text{int}(V^{-1}(0))$ is nonempty and has smooth boundary.
\item(S1):~ $0 < \gamma < 1$.
\end{enumerate}
%
%
%

The main result of this paper is as follows.


\begin{theorem}\label{main1}
	The problem \eqref{main_prob}
	\begin{enumerate}[label=(\roman*)]
\item admits $m$ distinct weak solutions with positive critical values, for any $\beta \geq\beta^*$.  Also, for $k\geq 2$, there exists infinitely many critical points $u$ with critical energy $c_k^*$, whereas, for $k=1$, there exists exactly two critical points, namely $u,-u$.
\item admits infinitely many solutions, irrespective of the value of the parameter $\beta>0$. 
\end{enumerate}
\end{theorem}

	
	
%
	\section{Mathematical preliminaries}

	In this section we will give the proof of the main results. Before that we introduce the function space in which solution(s) will be sought. This space is defined as follows.
	\begin{align}\label{SS}
	W_{\mu}:=\{u\in L^{p}(\Omega):\int_{\Omega}|\nabla u|^pdx+\mu\int_{\Omega}V(x)|u|^pdx<\infty\},
	\end{align}
	for $\mu>0$. Apparently, $W_{\mu}\subset W_0^{1,p}(\Omega)$ is equipped with the norm $$\|u\|_{\mu}^p=\int_{\Omega}(|\nabla u|^p+\mu V(x)|u|^p)dx$$ under which $W_{\mu}$ is a closed subspace and hence is reflexive.

	In addition to this, one can also see the existence of infinitely many solutions using the symmetric mountain pass theorem due to Kajikiya \cite{Ryuju1}. We at first recall the Krasnoselskii genus that forms an integral part of the symmetric mountain pass lemma. 
	\begin{definition}\label{genus}Let $X$ be a Banach space and $A\subset X$. $A$ is said to be symmetric if $u \in A$ implies $-u \in A$. For a closed symmetric set $A$ which does not contain the origin, we define a {\it genus} $\mathfrak{g}(A)$ of $A$ by the smallest integer $k$ such that there exists an odd continuous mapping from $A$ to $\mathbb{R}^k\setminus\{0\}$. If there does not exist such a $k$, we define $\mathfrak{g}(A)=\infty$. Moreover, we set $\mathfrak{g}(\emptyset) = 0$.
	\end{definition}

We now recall the following theorems due to Perera \cite{perera2} and Rabinowitz \cite{PR1}.
	\begin{theorem}\label{ACPT1}
	Let $X$ be a Banach space and $\mathcal{C}$ be a compact, symmetric subset of the unit sphere $S=\{u\in X:\|u\|_{1,p}=1\}$ with the cohomological index $i^*(\mathcal{C})=m\geq 1$. Let $ I \in C^1(X \setminus \{0\},\mathbb{R})$ and there exists $R>0$ such that
	\begin{align}\label{2.1}
		\begin{split}
			\underset{u\in \mathcal{A}}\sup~I(u)\leq 0,~~~\underset{u\in \mathcal{B}}\sup~I(u)< c^*,
		\end{split}
	\end{align}
	where $\mathcal{A}=\{Ru:u\in \mathcal{C}\}$, $\mathcal{B}=\left\{tu:u\in \mathcal{A},~0\leq t\leq 1\right\}$. Let $r\in(0,R)$ be so small that 
	\begin{align}\label{2.2}
		\begin{split}
			\underset{u\in S_r}\inf I(u)>0.
		\end{split}
	\end{align}
	Let $\mathcal{A}_j^*:=\{M\subset X:M~\text{is symmetric and}~i^*(M)\geq j\}$, and set $$c_j^*:=\underset{M\in \mathcal{A}_j^*}\inf\underset{u\in M}\sup~I(u), j=1,2,\cdots,m.$$
	Then $ 0 < c_1^*\leq c_2^*\leq\cdots\leq c_m^*<c^*$, each $c_j^*$ is a critical value of $I$ with $c_1^*<0$ and hence the functional $I$ has $m$ distinct associated critical points.
\end{theorem}

	\begin{theorem}{(Rabinowitz \cite[Theorem $9.12$]{PR1})}\label{symm_MP_Thm}
		Suppose that $\mathcal{E}:X\to\mathbb{R}$ and
		\begin{enumerate}[label=(\roman*)]
			\item $\mathcal{E}\in C^1(X,\mathbb{R})$ be even, $\mathcal{E}(0)=0$, and satisfies the Palais-Smale (PS) condition,
			\item $X=Y\bigoplus Z$, $Y$ being a finite dimensional subspace of $X$,
			\item there exists $\rho, \alpha>0$ such that $\mathcal{E}|_{\partial B_{\rho}\cap Z}\geq\alpha$,
			\item for each $W\subset X$, finite dimensional subspace, there exists $R_W>0$ such that $\mathcal{E}(u)\leq 0$ in $W\setminus B_{R_W}(0)$,
		\end{enumerate}
		Then $\mathcal{E}$ possess an unbounded sequence of critical values $(d_k)$, where  $d_k^*=\underset{A\in\Gamma_k}\inf\underset{u\in A}\sup\mathcal{E}(u)$, for $k>\text{dim}(Y)$, $$\Gamma_k=\{h(\overline{A_k\setminus V}):h\in G_k, \mathfrak{g}(V)\leq k-j, 0\leq j\leq k\},$$
		$G_k=\{h\in C(A_k,X):h~\text{is odd},~h|_{\partial B_{R_k}\cap E_k}=\text{id}\}$, $A_k=B_{R_k}\cap E_k$, $E_k=\text{span}\{e_1,e_2,\cdots,e_k\}$ for $R_k>0$ that has been so chosen that $\mathcal{E}|_{E_k\setminus B_{R_k}}\leq 0$.
	\end{theorem}
	 
	\begin{remark}
	 	We recall the definition of eigenvalues $\lambda_{m}$ from the Theorem $2.2$ \cite{perera3}.
	 \end{remark}	 
	 \noindent We define $$\psi(u)=\frac{1}{pH_{p}(u)}~~\text{and}~~\psi^{\lambda_m}=\{u\in S^{\infty}:\psi(u)\leq\lambda_m\},$$ where $$S^{\infty}=\{u\in X:\|u\|=1\}~~\text{and}~~H_{p}(u) = \frac{1}{p} \int_{\Omega} |u|^p \, dx.$$
	 Then, by the H\"{o}lder inequality, we have 
	 \begin{align}\label{H}
	 	\begin{split}
	 		\frac{1}{p^*} \int_{\Omega} |u|^{p^*}\, dx &\geq \frac{1}{p^*|\Omega|^{\frac{p^*}{p}-1}} \left(\int_{\Omega} |u|^p \, dx\right)^{\frac{p^*}{p}} =   \frac{d_2}{p^*} \left(pH_p(u)\right)^{\frac{p^*}{p}} = \frac{d_2}{p^*} \left(\frac{1}{\psi(u)}\right)^{\frac{p^*}{p}} \geq \frac{d_2}{p^*\lambda_m^{p^*/p}},
	 	\end{split}
	 \end{align}
	 where $d_2 = \frac{1}{|\Omega|^{\frac{p^*}{p}-1}}$.	 
	 We recall the following auxiliary result.
	 \begin{theorem}(Perera \cite[Theorem $1.3$]{perera2})\label{aux1}
	 	If $\lambda_m<\lambda_{m+1}$, then the set $\psi^{\lambda_m}$ has a compact symmetric subset $C$ of index $m$.
	 \end{theorem}	 
 
		\section{Proof of the main results}
	Indeed the functional we have in hand is not bounded below and hence we need to cleverly choose the critical point theorem that uses a different family of sets that is transferable to lower energy landscapes that lies on the other side of the higher energy mountain peaks around the origin. 
	
	\noindent We will now prove the main results of this paper.
%
%
		
	\begin{proof}[Proof of Theorem \ref{main1}]
We at first prove $(i)$.\\
	{\it proof of $(i)$}~~We will prove that the hypothesis of the abstract critical theorem by Perera \cite[Theorem $2.1$]{perera2} are satisfied by the energy functional $\mathcal{I}$. From Theorem $1.3$ in \cite{perera3} there exists $\mathcal{C}\subset \{u\in X:\|u\|=1\}$ such that $i^*(\mathcal{C})=m$. For $u\in\mathcal{A}$, i.e., $u=Rv$ where $v\in \mathcal{C}$, we have the following
	\begin{align}\label{eq2}
		\begin{split}
			\mathcal{I}(u)=&\frac{R^p}{p}\|v\|^p+\mu\frac{R^p}{p}\int_{\Omega}V(x)|v|^pdx-\beta\frac{R^{p^*}}{p^*}\int_{\Omega}|v|^{p^*}dx\\
			&-\frac{\lambda R^{1-\gamma}}{1-\gamma}\int_{\Omega}|v|^{1-\gamma}dx\leq 0,
		\end{split}
	\end{align}
	for a sufficiently large $R>0$. \begin{align}\label{hyp1}\underset{u\in\mathcal{A}}\sup~\mathcal{I}(u)\leq 0.\end{align}
	We now observe that for any $u\in W_{\mu}$ we have
	\begin{align}\label{eq1}
	\begin{split}
	\mathcal{I}(tu)=&\frac{t^pR^p}{p}\|u\|^p+\mu\frac{t^pR^p}{p}\int_{\Omega}V(x)|u|^pdx-\beta\frac{t^{p^*}R^{p^*}}{p^*}\int_{\Omega}|u|^{p^*}dx\\
	&-\frac{\lambda t^{1-\gamma}R^{1-\gamma}}{1-\gamma}\int_{\Omega}|v|^{1-\gamma}dx\\
	\leq & \frac{t^pR^p}{p}d_1-\beta\frac{t^{p^*}R^{p^*}}{p^*}\int_{\Omega}|u|^{p^*}dx\\
	\leq & \frac{t^pR^p}{p}d_1-d_2\beta\frac{t^{p^*}R^{p^*}}{p^*\lambda_m^{p^*/p}}\\
	=& \frac{\tau^p}{p}\lambda_m d_1-d_2\beta\frac{\tau^{p^*}}{p^*}=:\phi(\tau),
		\end{split}
	\end{align}
	where $\tau=\frac{tR}{\lambda_m^{1/p}}$, $d_1, d_2>0$ that resulted from the Sobolev space embedding. Indeed at $\tau=\left(\frac{\lambda_mc_1}{\beta}\right)^{\frac{1}{p^*-p}}=:\tau_*$ the function $\phi$ has a maximum value of $\phi(\tau_*)=\frac{(\lambda_mc_1)^{\frac{p^*}{p^*-p}}}{\beta^{\frac{p^*}{p^*-p}}d}<c^*$ for sufficiently large $\beta$, say for $\beta\geq\beta^*$. Thus \begin{align}\label{hyp2}\underset{u\in\mathcal{X}}\sup~\mathcal{I}(u)< c^*.\end{align}
	Also, there exists $\lambda, \rho>0$, sufficiently small, such that for any $u$ with $\|u\|=\rho$ we have
	\begin{align}\label{eq2'}
	\begin{split}
	\mathcal{I}(u)=&\frac{1}{p}\|u\|^p+\frac{\mu}{p}\int_{\Omega}V(x)|u|^pdx-\frac{\beta}{p^*}\int_{\Omega}|u|^{p^*}dx-\frac{\lambda}{1-\gamma}\int_{\Omega}|v|^{1-\gamma}dx\\
	\geq & \frac{1}{p}\|u\|^p-\frac{\beta}{p^*}\int_{\Omega}|u|^{p^*}dx-\frac{\lambda}{1-\gamma}\int_{\Omega}|u|^{1-\gamma}dx\\
	=&\frac{1}{p}\rho^p-d_3\frac{\beta}{p^*}\rho^{p^*}-d_4\frac{\lambda}{1-\gamma}\rho^{1-\gamma}>0.
	\end{split}
	\end{align}
Thus \begin{align}\label{hyp3}\underset{u\in S_r}\inf\mathcal{I}(u)>0.\end{align}

Furthermore, on employing the abstract critical point theorem we conclude that there exists $m$ distinct pairs of weak solutions. An interesting question here is that, can there exist more than one $u$ at each critical energy level?. The answer is {\it yes}, and the proof proceeds as follows:\\ We will first see that for any $c_k^*<c<c_{k+1}^*$, $i^*(I^{c})=k$. By the definition of $$c_k^*=\underset{B\in\mathcal{F}_k}\inf\underset{u\in B}\sup~I(u),$$
there exists $A\in \mathcal{F}_k$ such that $c_k^*<\underset{u\in A}\sup I(u)\leq c$. Thus $A\subset I^c$ and hence \begin{align}\label{index1}k\leq i^*(A)\leq i^*(I^c).\end{align}
We now claim that $I^c\notin \mathcal{F}_{k+1}$. For if it does, then $$c_{k+1}^*\leq \underset{u\in I^c}\sup~I(u)\leq c$$
that leads to a contradiction. Hence $I^c\notin \mathcal{F}_{k+1}$, i.e., \begin{align}\label{index2}i^*(I^c)\leq k.\end{align}
Therefore, $i^*(I^c)=k$.

Thus it can be concluded that there exists infinitely many critical points $u$ with critical energy $c_k^*$, whereas, for $k=1$, there exists exactly two critical points, namely $u,-u$.\\
{\it proof of $(ii)$}~~Let $W_{\mu}=W\bigoplus Z$, where $Z=\text{span}\{\phi_1,\phi_2,\cdots,\phi_m\}$ and $W$ is the closed complement of $Z$. The $m$ here is same as the index of the set $\mathcal{C}$ in $(i)$. The hypothesis $(i)-(iii)$ follows from the standard arguments in variational analysis and hence we are skipping the proof of it. Let $k\in\mathbb{N}$ and let us define $Z_k:=\text{span}\{\psi_1,\psi_2,\cdots,\psi_k\}$, $\psi_i$s being linearly independent collection of unit vectors of $W_{\mu}$. In other words, we choose an arbitrary finite dimensional space with $k>m$. We now consider a ball $B_{R_k}\subsetneq Z_k$ with $R_k$ sufficiently large. Therefore, for $u\in \partial B_{Z_k}(0,1)$ we have $\mathcal{I}(R u)<0$ for any $R\geq R_k$, due to the presence of the critical exponent. Note that, $B_{Z_k}(0,1)$ refers to the unit ball in $Z_k$. Thus for the arbitrary choice of $k>m$ there exists $Z_k\setminus B_{R_k}$ we have $\mathcal{I}|_{E_k\setminus B_{R_k}}\leq 0$. The hypotheses $(iv)$ also holds. Therefore, there exist infinitely many critical points with unbounded critical values.
\end{proof}

\subsection{Two independent sets of solutions obtained by two different methods}
Let $\Gamma_k^{R},~\Gamma_k^{K},~\Gamma_k^{AS}$ denote the sets with index $\geq k$ w.r.t. the Rabinowitz's theorem, the Krasnoselskii genus (see Definition \ref{genus}) and the Alexander-Spanier cohomological index. Apparently we observe that $d_k^*\geq c_k^*$ since $\Gamma_k^{R}\subsetneq\Gamma_k^{K}\subsetneq \Gamma_k^{AS}$ for every $k$. However, Rabinowitz's theorem uses sets that eventually turn out to be compact, symmetric sets in finite dimensional space whereas the sets considered in the abstract critical point theorem due to Perera \cite{perera2} is indeed compact, symmetric with the cohomological index being at least $k$. The thin difference is that in the latter case it is not merely restricted to being compact in finite dimensional spaces. We further note that the sets considered in $\Gamma_k^{R}$ have genus $\leq$ to a certain finite integer, whereas the sets in $\Gamma_k^{AS}$ are the ones $\geq$ to the same integer. Apparently, even if the critical values $c_k^*=d_k^*$, one may always choose $u_k\neq v_k$ such that  $u_k$ is obtained from the Rabinowitz's theorem and $v_k$ are obtained from the Perera's theorem.  This choice is rendered possible due to the fact that there exist more than one $u$ at each critical energy level as proved in Theorem \ref{main1} $(i)$.



Thus $c_{\mu}\leq \tilde{c}$ since $\tilde{E}\subset W_{\mu}(\Omega)$.
		\subsection{Regularity of solutions}
		We begin by showing that a solution to \eqref{main_prob} is {\it essentially bounded} in $\Omega$. We firstly approximate the problem \eqref{main_prob} via the following problems:
		\[
		\begin{cases}\label{main1_approx}
			-\Delta_{p} u = \beta |u|^{p^*-\epsilon-2}u + \lambda |u|^{-\gamma-1}u, & \text{in } \Omega, \\
			u = 0, & \text{on } \partial \Omega.
		\end{cases}
		\tag{$P_{\epsilon}$}
		\]
		We recall the couple of things that we already know, (a)~The problem \eqref{main_prob} has solution(s), (b)~The problem \eqref{main1_approx} also has solutions, say, $(u_{\epsilon})$.\\
		Thus for $\delta_n\to 0^+$, there exists a solution.
		
		Suppose $(u_{\epsilon})$ is unbounded in $W_{\mu}$. Then on testing the weak formulation of \eqref{main1_approx} with $u_{\epsilon}$ and considering the following yields 
		\begin{align}\label{Talenti0}
			\begin{split}
		c_{\epsilon}=I_{\epsilon}(u_{\epsilon})-\frac{1}{\alpha}\langle I_{\epsilon}'(u_{\epsilon}),u_{\epsilon}\rangle=&\left(\frac{1}{p}-\frac{1}{\alpha}\right)\|u_{\epsilon}\|_{\mu}^p-\beta\left(\frac{1}{p^*}-\frac{1}{\alpha}\right)\int_{\Omega}|u_{\epsilon}|^{p^*-\epsilon}dx\\
			&-\lambda\left(\frac{1}{1-\gamma}-\frac{1}{\alpha}\right)\int_{\Omega}|u_{\epsilon}|^{1-\gamma}dx\\
			\geq &\left(\frac{1}{p}-\frac{1}{\alpha}\right)\|u_{\epsilon}\|_{\mu}^p-\lambda\left(\frac{1}{1-\gamma}-\frac{1}{\alpha}\right)\int_{\Omega}|u_{\epsilon}|^{1-\gamma}dx
				\end{split}
				\end{align}
		We choose only those $u_{\epsilon}$ such that the corresponding critical values are bounded by $c$, a critical value $c:=\frac{1}{N} S^{N/p} $. This can be proved as follows:
		To compute the maximum energy along the mountain pass path $\max_{t \geq 0} I_\epsilon(t u_0)$, we expand the terms using Taylor series with respect to $\delta$ At first we consider $u_0=\varphi u_{\delta}$, where $u_{\delta}(x)=\left(\frac{\delta N\left(\frac{N-p}{p-1}\right)^{p-1}}{\delta^p+|x|^{\frac{p}{p-1}}}\right)^{\frac{N-p}{p}}$, $\varphi$ being a smooth cutoff function. It is a standard exercise to see that $\|\nabla u_\delta\|_p^p = \|u_\delta\|_{p^*}^{p^*} = S^{N/p}$. We now calculate the upper bound restriction on the energy levels so that we will be able to uniformly bound the sequence $(u_{\epsilon})$ in $W_{\mu}$. We observe that
		\begin{align}\label{talenti1}
		\begin{split}
		\|\nabla u_0\|_p^p + \mu \int_\Omega V(x)|u_0|^p \, dx = &S^{N/p} + O(\delta^{p})\\
		\int_\Omega |u_0|^{p^*-\epsilon} \, dx \approx & \int_\Omega |u_\delta|^{p^*-\epsilon} \, dx\\
		c_\epsilon \leq \max_{t \geq 0} I_\epsilon(t u_0) < &\frac{1}{N} S^{N/p} - C \cdot \epsilon^\alpha - \text{singular contributions}.
		\end{split}
		\end{align}
		The last step occurs by the following argument:
		\begin{align}\label{talenti2}
		\begin{split}
			\max_{t \geq 0} I_\epsilon(t u_0) \leq &\max_{t \geq 0} \left( \frac{t^p}{p} A - \frac{\beta t^{p^*-\epsilon}}{p^*} B_\epsilon \right) - \text{singular contributions}.
			\end{split}
			\end{align}
			Here $A, B_{\epsilon}$ denotes the first two terms in the functional $I_{\epsilon}$. On computing the maximum of the term in the bracket we get $$ t_\epsilon^{p^*-\epsilon-p} = \frac{p^* A}{\beta (p^*-\epsilon) B_\epsilon}.$$
			Evaluating the maximum of this algebraic structure yields:$$\max_{t \geq 0} \left( \frac{t^p}{p} A - \frac{\beta t^{p^*-\epsilon}}{p^*} B_\epsilon \right) = \left( \frac{1}{p} - \frac{1}{p^*-\epsilon} \right) \frac{A^{\frac{p^*-\epsilon}{p^*-\epsilon-p}}}{\left(\frac{\beta (p^*-\epsilon)}{p^*} B_\epsilon\right)^{\frac{p}{p^*-\epsilon-p}}}.$$
			Expanding this expression as a Taylor series around $\epsilon = 0$ gives:$$\max_{t \geq 0} I_\epsilon(t u_0) \leq \left( \frac{1}{p} - \frac{1}{p^*} \right) \frac{A^{p^*/p}}{B_0^{p/p^*}} - C \epsilon + O(\epsilon^2) - \text{singular contributions}$$

			Recall from our previous proofs that for the optimal Talenti bubble, $A = S^{N/p} + \mathcal{O}(\delta^p)$ and $B_0 = S^{N/p} + \mathcal{O}(\delta^{N-p})$. Plugging these exact values into the leading coefficient yields:$$\left( \frac{1}{p} - \frac{1}{p^*} \right) \frac{\left(S^{N/p} + O(\delta^p)\right)^{N/p}}{\left(S^{N/p} + O(\delta^{N-p})\right)^{(N-p)/p}} = \frac{1}{N} S^{N/p} + O(\delta^p).$$
			Substituting this back into our energy inequality gives:$$\max_{t \geq 0} I_\epsilon(t u_0) \leq \frac{1}{N} S^{N/p} + C_0 \delta^p - C_1 \epsilon - \text{singular contributions}$$
			
			We want to choose $\delta$ sufficiently small so that the truncation error $C_0\delta^p$ does not overwhelm the energy drop, but large enough to preserve the profile. We now make the optimal coupling choice:$$\delta = \epsilon^\theta \quad \text{where } \theta > \frac{1}{p}$$If we plug $\delta = \epsilon^\theta$ into the equation, the $C_1 \epsilon$ term becomes the dominant negative driver because its power of $\epsilon$ is lower than $p\theta$.
			
			Combining this with our asymptotic expansion yields the final result:$$c_\epsilon \leq \max_{t \geq 0} I_\epsilon(t u_0) < \frac{1}{N} S^{N/p} - C_1 \epsilon - \int_\Omega \frac{\lambda t_\epsilon^{1-\gamma}}{1-\gamma}|u_0|^{1-\gamma}dx.$$ On regrouping the remaining combined exponents which represents the linear order contribution from the shift in the subcritical exponent yields 
			$$\max_{t \geq 0} I_\epsilon(t u_0) \leq \frac{1}{N} S^{N/p} - C_1 \epsilon - \text{singular contributions}<\frac{1}{N} S^{N/p} - C_1 \epsilon.$$
			Thus $c_{\epsilon}$ are bounded by $\frac{1}{N} S^{N/p} $.
			
			Furthermore, from \eqref{Talenti0} we have $(u_{\epsilon})$ to be bounded in $W_{\mu}$ and hence by consequences of $W_{\mu}$ being reflexive leads us to conclude that $(u_{\epsilon})$ is bounded in measure upto a subsequence. The limit function $\tilde{u}$ is also $L^{\infty}$ bounded because
			\begin{align}\label{bdd0}
				\begin{split}
					\|\tilde{u}\|_{\infty}\leq&\|u_{\epsilon}-\tilde{u}\|_{\infty}\|+\|u_{\epsilon}\|_{\infty}\|\\
					<&\tilde{\eta}+M.
					\end{split}
					\end{align}
			\noindent{\bf H\"{o}lder continuous regularity}:~
			
	In what follows we restrict ourselves to positive weak solutions of \eqref{main_prob}. We now prove a weak Harnack type inequality for positive supersolutions of \eqref{main_prob} by applying the Moser iteration method. We firstly give an important remark.
	\begin{remark}\label{meas_0}
		The measure of the set $\{u=0\}$ is zero by arguments in Lemma $3$ of \cite{GS_FCAA} whenever $u$ is nonneagative.
	\end{remark}
			
			\begin{lemma}\label{lem3.1}
				Let $u$ be a positive supersolution of \eqref{main_prob}. Suppose $B_R(x_0)\Subset\Omega$. Then for every $\alpha>0$ and every
				\[
				\eta\in C_0^\infty(B_R(x_0)),
				\qquad
				0\le\eta\le1,
				\]
				there exists a positive constant $C_{12}=C_{12}(N,p,\alpha,\mu,\|V\|_\infty,R-r)>0$
				such that
				\begin{equation}\label{eq3.2}
					\inf_{B_r(x_0)}u
					\ge
					C_{12}
					\left(
					\int_{B_R(x_0)}
					u^{-\alpha}\,dx
					\right)^{-1/\alpha}.
				\end{equation}
			\end{lemma}
			
			\begin{proof}
				Since $u$ is a positive supersolution of \eqref{main_prob}, for every
				nonnegative function
				$\varphi\in W_0^{1,p}(\Omega)$, we have
				\begin{align}
					\int_{\Omega}
					|\nabla u|^{p-2}
					\nabla u\cdot\nabla\varphi\,dx
					+
					\mu
					\int_{\Omega}
					V(x)u^{p-1}\varphi\,dx
					\ge
					\beta
					\int_{\Omega}
					u^{p^*-1}\varphi\,dx
					+
					\lambda
					\int_{\Omega}
					u^{-\gamma}\varphi\,dx.
					\label{eq3.3}
				\end{align}				
				Let $m=\alpha+p-1$ and choose $\varphi=\eta^pu^{-m}$. It follows that
				\begin{align}
					\nabla\varphi =
					p\eta^{p-1}u^{-m}\nabla\eta
					- m\eta^pu^{-m-1}\nabla u.
					\label{eq3.4}
				\end{align}				
				Substituting \eqref{eq3.4} into \eqref{eq3.3}, we obtain
				\begin{align}
					&
					p
					\int_{\Omega}
					\eta^{p-1}
					u^{-m}
					|\nabla u|^{p-2}
					\nabla u\cdot\nabla\eta\,dx
					-
					m
					\int_{\Omega}
					\eta^pu^{-m-1}
					|\nabla u|^p\,dx
					+
					\mu
					\int_{\Omega}
					V(x)\eta^pu^{p-1-m}\,dx
					\nonumber\\
					&
					\qquad\ge
					\beta
					\int_{\Omega}
					\eta^pu^{p^*-1-m}\,dx
					+
					\lambda
					\int_{\Omega}
					\eta^pu^{-m-\gamma}\,dx.
					\label{eq3.5}
				\end{align}				
				Since the last two integrals are nonnegative, we deduce that
				\begin{align}
					m
					\int_{\Omega}
					\eta^pu^{-m-1}
					|\nabla u|^p\,dx
					\le
					p
					\int_{\Omega}
					\eta^{p-1}
					u^{-m}
					|\nabla u|^{p-1}
					|\nabla\eta|\,dx
					+
					\mu
					\int_{\Omega}
					V(x)\eta^pu^{p-1-m}\,dx.
					\label{eq3.6}
				\end{align}				
				Applying Young's inequality
				\[
				ab
				\le
				\varepsilon a^{\frac{p}{p-1}}
				+
				C(\varepsilon)b^p,
				\]
				with
				\[
				a=
				\eta^{p-1}
				u^{-\frac{(m+1)(p-1)}{p}}
				|\nabla u|^{p-1},
				\qquad
				b=
				pu^{\frac{p-1-m}{p}}
				|\nabla\eta|,
				\]
				we obtain
				\begin{align}
					&
					p
					\eta^{p-1}
					u^{-m}
					|\nabla u|^{p-1}
					|\nabla\eta|
					\le
					\varepsilon
					\eta^p
					u^{-m-1}
					|\nabla u|^p
					+
					C_1(\varepsilon)
					u^{p-1-m}
					|\nabla\eta|^p.
					\label{eq3.7}
				\end{align}	
				Note that $C_1(\varepsilon) = C(\varepsilon) p^p$. Integrating \eqref{eq3.7} over $\Omega$ and substituting into \eqref{eq3.6}, we get
				\begin{align}
					(m-\varepsilon)
					\int_{\Omega}
					\eta^p
					u^{-m-1}
					|\nabla u|^p\,dx
					\le
					C_1(\varepsilon)
					\int_{\Omega}
					u^{p-1-m}
					|\nabla\eta|^p\,dx
					+
					\mu
					\int_{\Omega}
					V(x)\eta^pu^{p-1-m}\,dx.
					\label{eq3.8}
				\end{align}				
				Since $m=\alpha+p-1$, we have $p-1-m=-\alpha$. Using (V1), we get 
				\begin{align}
					\mu
					\int_{\Omega}
					V(x)\eta^pu^{p-1-m}\,dx
					&\le
					\mu\|V\|_\infty
					\int_{\Omega}
					\eta^pu^{-\alpha}\,dx
					\nonumber\\
					&=:C_2
					\int_{\Omega}
					\eta^pu^{-\alpha}\,dx,
					\label{eq3.9}
				\end{align}
				where $C_2=\mu\|V\|_\infty$. Choose $\varepsilon=\frac{m}{2}$, then from \eqref{eq3.8} and \eqref{eq3.9}, we get
				\begin{align}
					\frac{m}{2}
					\int_{\Omega}
					\eta^pu^{-\alpha-p}
					|\nabla u|^p\,dx
					\le
					C_1\left(\frac m2\right)
					\int_{\Omega}
					u^{-\alpha}
					|\nabla\eta|^p\,dx
					+
					C_2
					\int_{\Omega}
					\eta^pu^{-\alpha}\,dx.
					\label{eq3.10}
				\end{align}				
				Finally, dividing by $\frac m2$ and setting
				\[
				C_3=\frac{2}{m}C_1\left(\frac m2\right),
				\qquad
				C_4=\max\left\{
				C_3,\,
				\frac{2C_2}{m}
				\right\},
				\]
				we conclude that
				\begin{align}
					\int_{\Omega}
					\eta^pu^{-\alpha-p}
					|\nabla u|^p\,dx
					\le
					C_4
					\left(
					\int_{\Omega}
					u^{-\alpha}
					|\nabla\eta|^p\,dx
					+
					\int_{\Omega}
					\eta^pu^{-\alpha}\,dx
					\right).
					\label{eq3.11}
				\end{align}
				Let $w=u^{-\frac{\alpha}{p}}$. Then,
				\begin{align}
					|\nabla w|^p =
					\left(\frac{\alpha}{p}\right)^p
					u^{-\alpha-p}
					|\nabla u|^p.
					\label{eq3.13}
				\end{align}			
				Multiplying both sides of \eqref{eq3.13} by $\eta^p$ and integrating over $\Omega$, we obtain
				\begin{align}
					\int_{\Omega}
					\eta^p
					|\nabla w|^p\,dx
					=
					\left(\frac{\alpha}{p}\right)^p
					\int_{\Omega}
					\eta^p
					u^{-\alpha-p}
					|\nabla u|^p\,dx.
					\label{eq3.14}
				\end{align}				
				Using \eqref{eq3.11}, we arrive at
				\begin{align}
					\int_{\Omega}
					\eta^p
					|\nabla w|^p\,dx
					\le
					C_5
					\left(
					\int_{\Omega}
					w^p
					|\nabla\eta|^p\,dx
					+
					\int_{\Omega}
					\eta^pw^p\,dx
					\right),
					\label{eq3.15}
				\end{align}
				where $C_5=\left(\frac{\alpha}{p}\right)^pC_4$.	Next, observe that
				\begin{align}
					|\nabla(\eta w)|^p
					\le
					2^{p-1}
					\eta^p|\nabla w|^p
					+
					2^{p-1}
					w^p|\nabla\eta|^p.
					\label{eq3.16}
				\end{align}				
				Integrating over $\Omega$, it follows that
				\begin{align}
					\int_{\Omega}
					|\nabla(\eta w)|^p\,dx
					\le
					2^{p-1}
					\int_{\Omega}
					\eta^p|\nabla w|^p\,dx
					+
					2^{p-1}
					\int_{\Omega}
					w^p|\nabla\eta|^p\,dx.
					\label{eq3.17}
				\end{align}		
				Substituting \eqref{eq3.15} into \eqref{eq3.17}, we obtain
				\begin{align}
					\int_{\Omega}
					|\nabla(\eta w)|^p\,dx
					\le
					C_6
					\int_{\Omega}
					w^p|\nabla\eta|^p\,dx
					+
					C_7
					\int_{\Omega}
					\eta^pw^p\,dx,
					\label{eq3.18}
				\end{align}
				where $C_6=2^{p-1}(C_5+1)$ and $C_7=2^{p-1}C_5$. Since $\eta w\in W_0^{1,p}(B_R(x_0))$, Sobolev's inequality yields
				\begin{align}
					\left(
					\int_{B_R(x_0)}
					|\eta w|^{p^*}\,dx
					\right)^{\frac{p}{p^*}}
					\le
					S
					\int_{B_R(x_0)}
					|\nabla(\eta w)|^p\,dx,
					\label{eq3.19}
				\end{align}
				where \(S\) denotes the Sobolev constant. Combining \eqref{eq3.18} and \eqref{eq3.19}, we obtain
				\begin{align}
					\left(
					\int_{B_R(x_0)}
					|\eta w|^{p^*}\,dx
					\right)^{\frac{p}{p^*}}
					\le
					C_8
					\int_{B_R(x_0)}
					w^p|\nabla\eta|^p\,dx
					+
					C_9
					\int_{B_R(x_0)}
					\eta^pw^p\,dx.
					\label{eq3.21}
				\end{align}	
			where $C_8=SC_6$ and $C_9=SC_7$. Choose $\eta\in C_0^\infty(B_R(x_0))$ such that $0\le \eta\le1 $, $\eta\equiv1$ on $B_r(x_0)$ and
				\[
				|\nabla\eta|
				\le
				\frac{1}{R-r}.
				\]			
				Since $\eta=1$ on $B_r(x_0)$ it follows from \eqref{eq3.21} that
				\begin{align}
					\left(
					\int_{B_r(x_0)}
					w^{p^*}\,dx
					\right)^{\frac{p}{p^*}}
					&\le
					C_8
					\int_{B_R(x_0)}
					w^p|\nabla\eta|^p\,dx
					+
					C_9
					\int_{B_R(x_0)}
					\eta^pw^p\,dx
					\nonumber\\
					&\le
					\frac{C_8}{(R-r)^p}
					\int_{B_R(x_0)}
					w^p\,dx
					+
					C_9
					\int_{B_R(x_0)}
					w^p\,dx.
					\label{eq3.22}
				\end{align}			
				Hence
				\begin{align}
					\left(
					\int_{B_r(x_0)}
					w^{p^*}\,dx
					\right)^{\frac{p}{p^*}}
					\le
					C_{10}
					\left(
					1+\frac1{(R-r)^p}
					\right)
					\int_{B_R(x_0)}
					w^p\,dx,
					\label{eq3.23}
				\end{align}
				where $C_{10}=\max\{C_8,C_9\}$. Since $w^{p^*}=u^{-\alpha\frac{p^*}{p}}$. Therefore \eqref{eq3.23} becomes
				\begin{align}
					\left(
					\int_{B_r(x_0)}
					u^{-\alpha\frac{p^*}{p}}
					\,dx
					\right)^{\frac{p}{p^*}}
					\le
					C_{10}
					\left(
					1+\frac1{(R-r)^p}
					\right)
					\int_{B_R(x_0)}
					u^{-\alpha}\,dx.
					\label{eq3.24}
				\end{align}			
				Let	$\chi=\frac{p^*}{p}=\frac{N}{N-p}$. Then			
				\begin{align}
					\left(
					\int_{B_r(x_0)}
					u^{-\alpha\chi}
					\,dx
					\right)^{\frac1{\alpha\chi}}
					\le
					C_{11}
					\left(
					1+\frac1{(R-r)^p}
					\right)^{\frac1\alpha}
					\left(
					\int_{B_R(x_0)}
					u^{-\alpha}\,dx
					\right)^{\frac1\alpha},
					\label{eq3.26}
				\end{align}
				where $C_{11}	=C_{10}^{1/\alpha}$. Define
				\[
				r_i
				=
				r+\frac{R-r}{2^i},
				\qquad i=0,1,2,\ldots.
				\]
				Then $r_0=R$ and $r_i\downarrow r$ as $i\to\infty$. Moreover,
				\[
				r_{i-1}-r_i
				=
				\frac{R-r}{2^i},
				\qquad i\ge1.
				\]		
				Next, define
				\[
				\alpha_i
				=
				\alpha\chi^i,
				\qquad i=0,1,2,\ldots,
				\]
				where $\chi=\frac{N}{N-p}>1$. Applying \eqref{eq3.26} with
				\[
				R=r_{i-1},
				\qquad
				r=r_i,
				\qquad
				\alpha=\alpha_{i-1},
				\]
				we obtain
				\begin{align}
					\left(
					\int_{B_{r_i}(x_0)}
					u^{-\alpha_i}\,dx
					\right)^{\frac1{\alpha_i}}
					\le
					C_{11}^{1/\chi^{\,i-1}}
					\left(
					1+
					\frac1{(r_{i-1}-r_i)^p}
					\right)^{\frac1{\alpha_{i-1}}}
					\left(
					\int_{B_{r_{i-1}}(x_0)}
					u^{-\alpha_{i-1}}\,dx
					\right)^{\frac1{\alpha_{i-1}}}.
					\label{eq3.27}
				\end{align}			
				Since $\frac1{(r_{i-1}-r_i)^p}=\frac{2^{ip}}{(R-r)^p}$. Therefore,
				\begin{align}
					\left(
					\int_{B_{r_i}(x_0)}
					u^{-\alpha_i}\,dx
					\right)^{\frac1{\alpha_i}}
					\le
					C_{11}^{1/\chi^{\,i-1}}
					\left(
					1+
					\frac{2^{ip}}{(R-r)^p}
					\right)^{\frac1{\alpha_{i-1}}}
					\left(
					\int_{B_{r_{i-1}}(x_0)}
					u^{-\alpha_{i-1}}\,dx
					\right)^{\frac1{\alpha_{i-1}}}.
					\label{eq3.28}
				\end{align}
				Iterating \eqref{eq3.28}, we obtain
				\begin{align}
					&
					\left(
					\int_{B_{r_i}(x_0)}
					u^{-\alpha_i}\,dx
					\right)^{\frac1{\alpha_i}}
					\le
					\prod_{j=1}^{i}
					C_{11}^{1/\chi^{j-1}}
					\prod_{j=1}^{i}
					\left(
					1+
					\frac{2^{jp}}{(R-r)^p}
					\right)^{1/\alpha_{j-1}}
					\left(
					\int_{B_R(x_0)}
					u^{-\alpha}\,dx
					\right)^{1/\alpha}.
					\label{eq3.29}
				\end{align}			
				Let
				\[
				A_i
				=
				\prod_{j=1}^{i}
				C_{11}^{1/\chi^{j-1}}, \qquad 
				B_i
				=
				\prod_{j=1}^{i}
				\left(
				1+
				\frac{2^{jp}}{(R-r)^p}
				\right)^{1/\alpha_{j-1}}.
				\]			
				Since $\sum_{j=1}^{\infty}\chi^{-(j-1)}=\frac{\chi}{\chi-1}$, it follows that
				\[
				A_i
				\longrightarrow
				A
				=
				C_{11}^{\frac{\chi}{\chi-1}}
				<\infty.
				\]			
			Next, since $\alpha_{j-1}=\alpha\chi^{j-1}$, it follows that
				\[
				\sum_{j=1}^{\infty}
				\frac1{\alpha_{j-1}}
				\ln\left(
				1+\frac{2^{jp}}{(R-r)^p}
				\right)
				<\infty,
				\]
				because $\sum_{j=1}^{\infty}j\chi^{-j}<\infty$. Consequently, $B_i\rightarrow B<\infty$. Passing to the limit in \eqref{eq3.29}, we obtain
				\begin{align}
					\lim_{i\to\infty}
					\left(
					\int_{B_{r_i}(x_0)}
					u^{-\alpha_i}\,dx
					\right)^{1/\alpha_i}
					\le
					AB
					\left(
					\int_{B_R(x_0)}
					u^{-\alpha}\,dx
					\right)^{1/\alpha}.
					\label{eq3.30}
				\end{align}			
				Since $r_i\downarrow r$ and $\alpha_i\rightarrow\infty$, we have
				\[
				\lim_{i\to\infty}
				\left(
				\int_{B_{r_i}(x_0)}
				u^{-\alpha_i}\,dx
				\right)^{1/\alpha_i}
				=
				\|u^{-1}\|_{L^\infty(B_r(x_0))}
				=
				\frac1{\inf_{B_r(x_0)}u}.
				\]			
				Hence, from \eqref{eq3.30},
				\[
				\frac1{\inf_{B_r(x_0)}u}
				\le
				AB
				\left(
				\int_{B_R(x_0)}
				u^{-\alpha}\,dx
				\right)^{1/\alpha}.
				\]			
				Finally, introducing $C_{12}=(AB)^{-1}$, we conclude that
				\begin{equation}
					\inf_{B_r(x_0)}u
					\ge
					C_{12}
					\left(
					\int_{B_R(x_0)}
					u^{-\alpha}\,dx
					\right)^{-1/\alpha}.
					\label{eq3.31}
				\end{equation}
				where $C_{12}=C_{12}(N,p,\alpha,\mu,\|V\|_\infty,R-r)>0$. This completes the proof of Lemma \ref{lem3.1}.
			\end{proof}

The following lemma shows the boundedness of weak solutions. The proof relies on a De Giorgi type iteration technique.

		\begin{lemma}\label{lem3.2}
			Suppose that $u\in W_0^{1,p}(\Omega)$ is a positive weak solution of \eqref{main_prob}. Then $u\in L^\infty(\Omega)$.
		\end{lemma}
		
		\begin{proof}
			For $k>0$, define
			\[
			A_k:=\{x\in\Omega:\ u(x)>k\},
			\qquad
			w_k:=(u-k)_+.
			\]			
			Since $w_k\in W_0^{1,p}(\Omega)$, taking $w_k$ as a test function in the
			weak formulation associated with \eqref{main_prob}, we obtain			
			\begin{align}
				\int_{A_k} |\nabla w_k|^p\,dx
				+
				\mu\int_{A_k}V(x)u^{p-1}w_k\,dx
				=
				\beta\int_{A_k}u^{p^*-1}w_k\,dx
				+
				\lambda\int_{A_k}u^{-\gamma}w_k\,dx .
				\label{eq1}
			\end{align}			
			Since $V\ge0$, we may drop the nonnegative potential term and get			
			\begin{equation}
				\int_{A_k} |\nabla w_k|^p\,dx
				\le
				\beta\int_{A_k}u^{p^*-1}w_k\,dx
				+
				\lambda\int_{A_k}u^{-\gamma}w_k\,dx .
				\label{eq2}
			\end{equation}			
			Since $w_k\le u$ on $A_k$, we have		
			\[
			u^{p^*-1}w_k
			\le
			u^{p^*},
			\]			
			and therefore			
			\begin{equation}
				\int_{A_k}u^{p^*-1}w_k\,dx
				\le
				\int_{A_k}u^{p^*}\,dx .
				\label{eq3}
			\end{equation}			
			Moreover, $u=w_k+k$ on $A_k$ . Hence, 		
			\begin{align}
				\int_{A_k}u^{p^*}\,dx
				&=
				\int_{A_k}(w_k+k)^{p^*}\,dx
				\nonumber\\
				&\le
				2^{p^*-1}
				\left(
				\int_{A_k}w_k^{p^*}\,dx
				+
				k^{p^*}|A_k|
				\right).
				\label{eq4}
			\end{align}			
			Next, since $u>k$ on $A_k$ and $0<\gamma<1$,			
			\[
			u^{-\gamma}
			\le
			k^{-\gamma}.
			\]			
			Consequently,			
			\begin{equation}
				\int_{A_k}u^{-\gamma}w_k\,dx
				\le
				k^{-\gamma}
				\int_{A_k}w_k\,dx .
				\label{eq5}
			\end{equation}			
			Substituting \eqref{eq4} and \eqref{eq5} into \eqref{eq2}, we obtain			
			\begin{align}
				\int_{A_k} |\nabla w_k|^p\,dx
				\le
				D_1
				\left(
				\int_{A_k}w_k^{p^*}\,dx
				+
				k^{p^*}|A_k|
				\right)
				+
				D_2 k^{-\gamma}
				\int_{A_k}w_k\,dx ,
				\label{eq6}
			\end{align}
			for some constants $D_1,D_2>0$.			
			Applying Hölder's inequality,			
			\[
			\int_{A_k}w_k\,dx
			\le
			|A_k|^{1-\frac1{p^*}}
			\left(
			\int_{A_k}w_k^{p^*}\,dx
			\right)^{\frac1{p^*}}.
			\]			
			Therefore			
			\[
			k^{-\gamma}\int_{A_k}w_k\,dx
			\le
			k^{-\gamma}
			|A_k|^{1-\frac1{p^*}}
			\left(
			\int_{A_k}w_k^{p^*}\,dx
			\right)^{\frac1{p^*}}.
			\]			
			Using Young's inequality, we infer that			
			\[
			k^{-\gamma}\int_{A_k}w_k\,dx
			\le
			D_3
			\int_{A_k}w_k^{p^*}\,dx
			+
			D_4
			k^{-\frac{\gamma p^*}{p^*-1}}
			|A_k|,
			\]			
			for suitable constants $D_3,D_4>0$.			
			Substituting into \eqref{eq6}, we obtain			
			\begin{align}
				\int_{A_k} |\nabla w_k|^p\,dx
				\le
				D_5
				\left(
				\int_{A_k}w_k^{p^*}\,dx
				+
				k^{p^*}|A_k|
				+
				k^{-\frac{\gamma p^*}{p^*-1}}
				|A_k|
				\right).
				\label{eq7}
			\end{align}			
			Since $k\ge1$ implies			
			\[
			k^{-\frac{\gamma p^*}{p^*-1}}
			\le
			1
			\le
			k^{p^*},
			\]			
			the last term can be absorbed into the term
			$k^{p^*}|A_k|$. Hence there exists $D_6>0$ such that			
			\begin{equation}
				\int_{A_k} |\nabla w_k|^p\,dx
				\le
				D_6
				\left(
				\int_{A_k}w_k^{p^*}\,dx
				+
				k^{p^*}|A_k|
				\right).
				\label{eq8}
			\end{equation}			
			By Sobolev's inequality,			
			\[
			S
			\left(
			\int_{A_k}w_k^{p^*}\,dx
			\right)^{\frac{p}{p^*}}
			\le
			\int_{A_k}|\nabla w_k|^p\,dx .
			\]			
			Combining this with \eqref{eq8}, we arrive at			
			\begin{equation}\label{eq9}
				\left(
				\int_{A_k}(u-k)^{p^*}\,dx
				\right)^{\frac{p}{p^*}}
				\le
				D_7
				\left[
				\int_{A_k}(u-k)^{p^*}\,dx
				+
				k^{p^*}|A_k|
				\right]
			\end{equation}
			for some constant $D_7>0$. For $M>2K$, define		
		\[
		k_n:=M\left(1-\frac1{2^n}\right),
		\qquad n\in\mathbb N\cup\{0\}.
		\]		
		Since \eqref{eq9} is precisely of the form required in \cite[Lemma~3.2]{CGL}
		(with $p$ replaced by $p^*$ and $r=\frac{p^*}{p}>1$),
		repeating verbatim the proof of \cite[Lemma~3.2]{CGL}, we infer that		
		\begin{equation}\label{eq10}
			\|(u-k_n)_+\|_{L^{p^*}(A_{k_n})}
			\longrightarrow 0
			\qquad\text{as } n\to\infty,
		\end{equation}		
		provided		
		\begin{equation}\label{eq11}
			\|(u-M/2)_+\|_{L^{p^*}(A_{M/2})}
		\end{equation}
		is sufficiently small.		
		Since $u\in L^{p^*}(\Omega)$, we have		
		\[
		\|(u-M/2)_+\|_{L^{p^*}(A_{M/2})}
		=
		\left(
		\int_\Omega (u-M/2)_+^{p^*}\,dx
		\right)^{1/p^*}
		\longrightarrow 0
		\]		
		as $M\to\infty$. Hence $M$ can be chosen large enough so that
		\eqref{eq11} holds.		
		Keeping such an $M$ fixed and using \eqref{eq10}, we obtain		
		\[
		\int_\Omega (u-M)_+^{p^*}\,dx
		\le
		\int_\Omega (u-k_n)_+^{p^*}\,dx
		=
		\int_{A_{k_n}} (u-k_n)_+^{p^*}\,dx
		\longrightarrow 0 .
		\]		
		Therefore		
		\[
		\int_\Omega (u-M)_+^{p^*}\,dx=0,
		\]		
		which implies		
		\[
		(u-M)_+=0
		\qquad\text{a.e. in }\Omega.
		\]		
		Hence		
		\[
		u\le M
		\qquad\text{a.e. in }\Omega.
		\]		
		Consequently,		
		\[
		u\in L^\infty(\Omega).
		\]		
		This completes the proof.
	\end{proof}

		\begin{lemma}\label{rhsbounded}
			Let $u$ be a positive weak solution of \eqref{main_prob}. Then
			\[
			f(x):=
			\beta u^{p^*-1}
			+\lambda u^{-\gamma}
			-\mu V(x)u^{p-1}
			\in L^\infty_{\rm loc}(\Omega).
			\]
		\end{lemma}
		
		\begin{proof}
			By Lemma \ref{lem3.2}, $u\in L^\infty(\Omega)$. Hence there exists a constant $M>0$ such that
			\[
			0<u(x)\le M
			\qquad \text{for a.e. }x\in\Omega.
			\]			
			Let $B_{2R}(x_0)\Subset\Omega$.	Since $u$ is a positive supersolution, Lemma \ref{lem3.1}
			implies that
			\[
			\inf_{B_R(x_0)}u
			\ge
			C_{12}
			\left(
			\int_{B_{2R}(x_0)}
			u^{-\alpha}\,dx
			\right)^{-1/\alpha}.
			\]			
			Therefore there exists a constant
			$m_R>0$ such that
			\[
			u(x)\ge m_R
			\qquad
			\text{for all }
			x\in B_R(x_0).
			\]			
			Consequently,
			\[
			u^{-\gamma}(x)
			\le
			m_R^{-\gamma}
			\qquad
			\text{for all }
			x\in B_R(x_0),
			\]
			which yields
			\[
			u^{-\gamma}
			\in
			L^\infty(B_R(x_0)).
			\]			
			Since $B_R(x_0)\Subset\Omega$ is arbitrary,
			\begin{align}\label{sing_bound}
				u^{-\gamma}
				\in
				L^\infty_{\mathrm{loc}}(\Omega).
			\end{align}			
			Next, since $u\in L^\infty(\Omega)$, we obtain
		\begin{align}\label{cri_bound}
			u^{p^*-1}
			\in
			L^\infty(\Omega).
		\end{align}
 By Lemma \ref{lem3.2} and assumption (V1), we get 
        \begin{align}\label{V_bound}
        	V(x)u^{p-1}
        	\in
        	L^\infty(\Omega).
        \end{align}			
			Combining \eqref{sing_bound}, \eqref{cri_bound} and \eqref{V_bound},
			we conclude that
			\[
			\beta u^{p^*-1}
			+\lambda u^{-\gamma}
			-\mu V(x)u^{p-1}
			\in
			L^\infty_{\mathrm{loc}}(\Omega).
			\]			
			Hence
			\[
			f\in L^\infty_{\mathrm{loc}}(\Omega).
			\]
		\end{proof}
	
	\begin{theorem}\label{holderregularity}
		Let $u$ be a positive weak solution of \eqref{main_prob}. Then there exists
		$\alpha\in(0,1)$ such that
		\[
		u\in C^{1,\alpha}_{\mathrm{loc}}(\Omega).
		\]
	\end{theorem}
	
	\begin{proof}
		By Lemma \ref{rhsbounded},
		\[
		f(x):=\beta u^{p^*-1}+\lambda u^{-\gamma}-\mu V(x)u^{p-1}
		\in L^\infty_{\mathrm{loc}}(\Omega).
		\]
		Hence $u$ satisfies
		\[
		-\Delta_p u=f(x)
		\qquad \text{in } \Omega,
		\]
		with $f\in L^\infty_{\mathrm{loc}}(\Omega)$. Therefore, by the local regularity theorem of Tolksdorf \cite{Tolksdorf1984} (see also Lieberman \cite{Lieberman1988,Lieberman1991}), there exists $\alpha\in(0,1)$ such that
		\[
		u\in C^{1,\alpha}_{\mathrm{loc}}(\Omega).
		\]
	\end{proof}

	\section*{Competing Interest}
	There are no competing interests.
	
	\section*{Acknowledgement}	
	The author thanks Dr. Aditi Chakrabarty of the Karlsruhe Institute of Technology, Germany.
	
	\end{document}